\documentclass[11pt]{article}
\PassOptionsToPackage{obeyspaces}{url}

\usepackage[T1]{fontenc}
\usepackage{lmodern}
\usepackage{amsmath,amssymb,amsthm,mathtools}
\usepackage{enumitem,microtype,etoolbox}
\usepackage{tikz}
\usepackage{aliascnt}
\usepackage[hidelinks]{hyperref}
\usepackage{bookmark}
\bookmarksetup{open,numbered,depth=5}
\hypersetup{
  pdftitle={Independent Sets and Balanced Cycle-Linkings in 2-Connected Graphs},
  pdfauthor={Shuaichao Wang}
}

\allowdisplaybreaks[4]

\newtheorem{theorem}{Theorem}[section]

\newaliascnt{lemma}{theorem}
\newtheorem{lemma}[lemma]{Lemma}
\aliascntresetthe{lemma}

\newaliascnt{proposition}{theorem}
\newtheorem{proposition}[proposition]{Proposition}
\aliascntresetthe{proposition}

\newaliascnt{corollary}{theorem}
\newtheorem{corollary}[corollary]{Corollary}
\aliascntresetthe{corollary}

\theoremstyle{definition}

\newaliascnt{remark}{theorem}
\newtheorem{remark}[remark]{Remark}
\aliascntresetthe{remark}

\newaliascnt{problem}{theorem}
\newtheorem{problem}[problem]{Problem}
\aliascntresetthe{problem}

\usepackage[capitalise,nameinlink]{cleveref}

\newcommand{\Turan}{\operatorname{Tur}}
\newcommand{\cM}{\mathcal{M}}
\newcommand{\floor}[1]{\left\lfloor #1\right\rfloor}
\newcommand{\ceil}[1]{\left\lceil #1\right\rceil}
\newcommand{\turcl}[3]{\mathsf{T}_{#1,#2}(#3)}

\newcommand{\preceqcf}{\preceq_{\mathrm{cf}}}
\newcommand{\doi}[1]{\href{https://doi.org/#1}{doi:#1}}

\patchcmd{\thebibliography}
  {\leftmargin\labelwidth}
  {\leftmargin\labelwidth\addtolength\itemsep{-0.3\baselineskip}}{}{}

\title{Independent Sets and Balanced Cycle-Linkings in 2-Connected Graphs}

\author{Shuaichao Wang\thanks{Faculty of Mathematics and Statistics,
Central China Normal University, Wuhan 430079, PR China.
\texttt{shuaichao@ccnu.edu.cn}.}}

\date{}

\begin{document}

% Preserve the reference order of the supplied manuscript.
\nocite{AlonShikhelman2016,Zykov1949,BougardJoret2008,DeLucaZamboni2021,FranklFurediKalai1988,Frohmader2008,LehnerWagner2017}

\maketitle

\begin{abstract}
For every $\alpha\ge3$ and all sufficiently large $n$, we identify a
single graph that simultaneously maximizes the number of independent
sets of every size among all $n$-vertex $2$-connected graphs with
independence number $\alpha$. This graph is unique up to isomorphism
and is a balanced cycle-linking of the disjoint union of $\alpha$
cliques whose orders differ by at most one. More precisely, for each
$3\le\beta\le\alpha$, the graphs maximizing the number of independent
$\beta$-sets are exactly the cycle-linkings whose clique-size cyclic
words are $\lfloor\beta/2\rfloor$-balanced. These results extend the
corresponding extremal result for connected graphs to the
$2$-connected setting. The proof combines generalized Tur\'an-type
clique counting and the edge-extremal theory of $2$-connected graphs
with a coefficientwise balancing-switch argument. The switch also
shows that a shortest imbalance of length $r$ first affects the
independent-set count in degree $2r$.
\end{abstract}

\section{Introduction}
The classical Tur\'an problem asks for the maximum number of edges
in an $n$-vertex graph containing no copy of a prescribed graph $F$;
see \cite{Bollobas1978} for background. For complete forbidden
graphs, the exact extremal number was determined by Tur\'an
\cite{Turan1941}. Its generalized form asks for the maximum number
$\operatorname{ex}(n,H,F)$ of copies of a fixed graph $H$ in an
$n$-vertex $F$-free graph. The systematic study of this parameter
was initiated by Alon and Shikhelman \cite{AlonShikhelman2016}.
The clique-counting case goes back to Zykov \cite{Zykov1949}:
for $2\le\beta\le\alpha$, the balanced complete $\alpha$-partite
graph maximizes the number of copies of $K_\beta$ among all
$n$-vertex $K_{\alpha+1}$-free graphs.

All graphs in this paper are finite and simple. For a graph $G$,
write $e(G)$ for its number of edges, $\alpha(G)$ for its independence
number, and $i_j(G)$ for its number of independent sets of
cardinality $j$. Passing to complements gives the corresponding
independent-set problem: if $H=\overline G$, then
$i_j(G)=k_j(H)$ and $\alpha(G)=\omega(H)$, where $k_j(H)$ counts
copies of $K_j$ and $\omega(H)$ is the clique number of $H$.
Following Bougard--Joret and Lehner--Wagner, let $T(n,\alpha)$
denote the disjoint union of $\alpha$ cliques whose orders differ
by at most one. Zykov's theorem therefore implies that, without a
connectivity requirement, $T(n,\alpha)$ simultaneously maximizes
every $i_j$ among graphs of order $n$ and independence number
$\alpha$.

For $1\le\alpha<n$, the corresponding connected construction is the
\emph{Tur\'an connected graph} $TC_{n,\alpha}$. It is obtained from
$T(n,\alpha)$ by choosing a vertex in a largest clique and joining
it to one vertex in every other clique. Bruy\`ere and M\'elot
\cite{BruyereMelot2009} proved that $TC_{n,\alpha}$ maximizes the
total number of independent sets among connected graphs of order
$n$ and independence number $\alpha$. Lehner and Wagner
\cite[Theorem~2]{LehnerWagner2017} strengthened this result by
showing that the same graph simultaneously maximizes $i_j$ for
every $1\le j\le\alpha$. For maximum independent sets, Mohr and
Rautenbach \cite{MohrRautenbach2021} characterized all extremal
connected graphs with prescribed order and independence number.

We seek a $2$-connected analogue of the fixed-size result.
Equivalently, we consider clique counting in $K_{\alpha+1}$-free
graphs with clique number $\alpha$ whose complements are
$2$-connected. For $j=2$, the identity
$i_2(G)=\binom n2-e(G)$ turns the problem into minimizing the
number of edges. Bougard and Joret determined this minimum and
all equality cases \cite[Proposition~6]{BougardJoret2008}.
For $\alpha\ge3$ and $n\ge3\alpha$, the extremal graphs are
precisely the \emph{cycle-linkings} of $T(n,\alpha)$: arrange the
clique components cyclically and join each pair of consecutive
cliques by one edge, using distinct attachment vertices for the
two linking edges incident with each clique.

The minimum-edge classification does not, however, determine the
higher independent-set counts. Although all these cycle-linkings
have the same number of edges, the cyclic order of the clique sizes
can affect $i_j$ when $j\ge4$. Thus two tasks remain: to show that
higher independent-set counts are maximized within the minimum-edge
family, and to determine the optimal cyclic orders within that
family.

The clique orders in $T(n,\alpha)$ are already as equal as possible;
the remaining choice is how to distribute the larger cliques around
the cycle. Write $n=q\alpha+\rho$, where $0\le\rho<\alpha$.
Encode the $\rho$ cliques of order $q+1$ by $1$ and the remaining
cliques by $0$; when $\rho=0$, use the all-zero word.
A cyclic word is considered up to rotation, and its cyclic factors
are read in its periodic extension. It is \emph{balanced} if any
two cyclic factors of the same length contain numbers of $1$'s
differing by at most one. For background on balanced and mechanical
words, see \cite[Chapter~2]{Lothaire2002}.

For each $0\le\rho<\alpha$, a balanced word of length $\alpha$
with $\rho$ ones is given by the $\alpha$-periodic sequence
\begin{equation}\label{eq:mechanical-word}
 b_i=\left\lfloor\frac{(i+1)\rho}{\alpha}\right\rfloor
     -\left\lfloor\frac{i\rho}{\alpha}\right\rfloor,
 \qquad i\in\mathbb Z.
\end{equation}
Each period contains $\rho$ ones, and every length-$s$ factor
contains either $\lfloor s\rho/\alpha\rfloor$ or
$\lceil s\rho/\alpha\rceil$ ones. The balanced cyclic word with
these prescribed numbers is unique up to rotation; see
\cite[Proposition~4.4 and the preceding discussion]{DeLucaZamboni2021}.
The constant-word case is immediate. For $\alpha\ge3$ and
$n\ge2\alpha$, let $B(n,\alpha)$ denote the corresponding balanced
cycle-linking of $T(n,\alpha)$, which is well defined up to
isomorphism.

The \emph{independence polynomial} of $G$ is
\[
 I(G;x):=\sum_{j\ge0}i_j(G)x^j.
\]
For polynomials $f,g\in\mathbb Z[x]$, write $f\preceqcf g$ if every
coefficient of $f$ is at most the corresponding coefficient of $g$.
Our main result identifies a single graph maximizing all these
coefficients simultaneously in the $2$-connected setting.

\begin{theorem}\label{thm:main-coefficientwise}
For every $\alpha\ge3$, there exists $N_\alpha$ such that, whenever
$n\ge N_\alpha$, every $n$-vertex $2$-connected graph $G$ with
$\alpha(G)=\alpha$ satisfies
\[
 I(G;x)\preceqcf I(B(n,\alpha);x).
\]
Equality holds if and only if $G\cong B(n,\alpha)$.
\end{theorem}

Uniqueness in \cref{thm:main-coefficientwise} concerns equality of
the entire independence polynomial. For an individual coefficient,
additional maximizers may occur. To describe them, call a cyclic
binary word \emph{$t$-balanced}, for an integer $t\ge1$, if the
balance condition holds for every factor length at most $t$.
For $3\le\beta\le\alpha$, let $\mathcal B_\beta(n,\alpha)$ be the
family of cycle-linkings of $T(n,\alpha)$ whose clique-size words
are $\lfloor\beta/2\rfloor$-balanced. This family is understood to
be closed under isomorphism.

\begin{theorem}\label{thm:fixed-size}
For every $3\le\beta\le\alpha$, there exists $N_{\alpha,\beta}$
such that, whenever $n\ge N_{\alpha,\beta}$, every $n$-vertex
$2$-connected graph $G$ with $\alpha(G)=\alpha$ satisfies
\[
 i_\beta(G)\le i_\beta(B(n,\alpha)).
\]
Equality holds if and only if $G\in\mathcal B_\beta(n,\alpha)$.
\end{theorem}
The cyclic-order comparison is related to the block-reversal method
of De Luca and Zamboni
\cite[Propositions~3.1 and~4.4]{DeLucaZamboni2021}
for extremal cyclic continuants. Their work connects reversal
comparisons with balanced binary cyclic words. Here we obtain
coefficientwise control of the corresponding independence
polynomials. More precisely, let $C_q(w)$ denote the cycle-linking
with successive clique orders $q+w_0,\ldots,q+w_{\alpha-1}$.
For every $q\ge2$, if $w$ is not balanced and its shortest imbalance
has length $r$, then \cref{thm:coefficientwise-order} shows that
$i_j(C_q(w))=i_j(B(n,\alpha))$ for $j<2r$, whereas
$i_j(C_q(w))<i_j(B(n,\alpha))$ for $2r\le j\le\alpha$.
Thus the shortest imbalance determines exactly the first coefficient
that falls below the balanced value, explaining the scale
$\lfloor\beta/2\rfloor$ in \cref{thm:fixed-size}.
\begin{remark}\label{rem:thresholds}
In \cref{thm:fixed-size}, one may take
$
N_{\alpha,\beta}=3\alpha(\beta-2)(\alpha+1).
$
Taking the maximum over $3\le\beta\le\alpha$ gives the sufficient
threshold
$
N_\alpha=3\alpha(\alpha-2)(\alpha+1)
$
for \cref{thm:main-coefficientwise}.
Thus $N_{\alpha,\beta}=O_\beta(\alpha^2)$ for fixed $\beta$,
whereas $N_\alpha=O(\alpha^3)$.
These thresholds are not claimed to be optimal.
For $G\in\mathcal B_\beta(n,\alpha)$, the extremal value
$i_\beta(G)$ can be evaluated using \cref{prop:matching-formula}.
\end{remark}

\paragraph{Our approach.}
The proof has two components. Within the cycle-linking family, a
shortest imbalance in the clique-size word yields a two-edge switch.
An exact polynomial identity shows that an imbalance of length $r$
produces a strict improvement in precisely the degrees from $2r$ to
$\alpha$, while the switch preserves the weight multisets of all
shorter cyclic factors. Iterating these switches gives the balanced
extremal order and the complete fixed-size equality characterization.
For the global reduction, Frohmader--Frankl--F\"uredi--Kalai clique
counting shows that, for sufficiently large $n$, any graph with more
edges than the Bougard--Joret minimum has fewer independent
$\beta$-sets than a cycle-linking. The edge-extremal classification
then reduces the problem to the cyclic-order comparison.

\section{Preliminary}
\label{sec:tools}

\subsection{Clique counts in balanced multipartite graphs}

For integers $r\ge1$ and $m\ge0$, let $\Turan_r(m)$ be the balanced
complete $r$-partite graph on $m$ vertices, allowing empty parts.
Write
\[
 \turcl{r}{k}{m}:=k_k(\Turan_r(m))=\binom mk_r.
\]
We use $\turcl{r}{0}{m}=1$ and $\turcl{r}{k}{m}=0$ when $k>r$ or
$m<k$. Since $\overline{T(n,\alpha)}=\Turan_\alpha(n)$, setting
$t(n,\alpha):=e(T(n,\alpha))$ gives
\begin{equation}\label{eq:T-independent}
 i_\beta(T(n,\alpha))=\turcl{\alpha}{\beta}{n},
 \qquad
 \turcl{\alpha}{2}{n}=\binom n2-t(n,\alpha).
\end{equation}

\begin{lemma}\label{lem:turan-recurrence}
Let $r\ge2$, $1\le k\le r$, and $m\ge1$. Then
\begin{equation}\label{eq:general-recurrence}
 \turcl{r}{k}{m}-\turcl{r}{k}{m-1}
 =\turcl{r-1}{k-1}{m-\ceil{m/r}}.
\end{equation}
\end{lemma}

\begin{proof}
Insert a vertex into a smallest part of $\Turan_r(m-1)$. Its
neighborhood is a balanced complete $(r-1)$-partite graph on
$m-\lceil m/r\rceil$ vertices, and each new $K_k$ consists of this
vertex and a $K_{k-1}$ in its neighborhood.
\end{proof}

\subsection{A Frohmader--FFK bound}

For a simplicial complex $\Delta$, let $c_j(\Delta)$ count its faces
of cardinality $j$. The complex is $r$-colorable if its vertices can
be colored with $r$ colors so that every face has distinct colors.
The clique complex of a graph consists of its cliques and is a flag
complex. Frohmader's theorem \cite[Theorem~1.1]{Frohmader2008}
implies that the clique complex of a $K_{r+1}$-free graph has the same
face numbers as an $r$-colorable complex. We combine this with the
canonical form of the Frankl--F\"uredi--Kalai theorem
\cite{FranklFurediKalai1988}; see
\cite[Lemma~3.6 and Theorem~3.7]{Frohmader2008}.

\begin{theorem}[Frohmader {\cite[Theorem~1.1]{Frohmader2008}}]
For every flag simplicial complex $\Delta$, there exists a balanced
simplicial complex $\Gamma$ with the same face vector; that is,
\[
 c_j(\Gamma)=c_j(\Delta)
 \qquad\text{for every }j\ge0.
\]
\end{theorem}

Recall that, for $r\ge k$, every positive integer $m$ has a unique
$(k,r)$-canonical representation
\[
 m=
 \binom{n_k}{k}_r+
 \binom{n_{k-1}}{k-1}_{r-1}+\cdots+
 \binom{n_{k-h}}{k-h}_{r-h},
\]
where
\[
 n_{k-i}-\floor{n_{k-i}/(r-i)}>n_{k-i-1}
 \quad(0\le i<h),
 \qquad
 n_{k-h}\ge k-h>0;
\]
see \cite[Lemma~3.6]{Frohmader2008}.

\begin{theorem}[Frankl--F\"uredi--Kalai
{\cite{FranklFurediKalai1988}; see also
\cite[Theorem~3.7]{Frohmader2008}}]\label{thm:ffk}
Let $1\le k\le r$, and let $\Delta$ be an $r$-colorable
simplicial complex. If
\[
 c_k(\Delta)=
 \binom{n_k}{k}_r+
 \binom{n_{k-1}}{k-1}_{r-1}+\cdots+
 \binom{n_{k-h}}{k-h}_{r-h}
\]
is the $(k,r)$-canonical representation of $c_k(\Delta)$, then
\[
 c_{k+1}(\Delta)\le
 \binom{n_k}{k+1}_r+
 \binom{n_{k-1}}{k}_{r-1}+\cdots+
 \binom{n_{k-h}}{k-h+1}_{r-h}.
\]
\end{theorem}

The following iterated consequence is the form needed later.
\begin{lemma}\label{lem:general-ffk}
Let $r\ge\beta\ge2$ and $N\ge\beta$. If $H$ is $K_{r+1}$-free and
\[
 e(H)\le\turcl{r}{2}{N}+s,
 \qquad 0\le s<N-\floor{N/r},
\]
then
\begin{equation}\label{eq:general-ffk}
 k_\beta(H)\le\turcl{r}{\beta}{N}
                  +\turcl{r-1}{\beta-1}{s}.
\end{equation}
\end{lemma}

\begin{proof}
Frohmader's theorem gives an $r$-colorable complex $\Delta$ with
$c_j(\Delta)=k_j(H)$ for every $j$; unused colors are allowed if its
dimension is smaller than $r-1$. For $2\le j\le\beta$, put
\[
 B_j:=\binom Nj_r+\binom{s}{j-1}_{r-1}.
\]
We prove $c_j(\Delta)\le B_j$ by induction on $j$. The case $j=2$
is the assumed edge bound, since $\binom{s}{1}_{r-1}=s$.

Suppose that $c_j(\Delta)\le B_j$ for some $2\le j<\beta$, and let
$h_j=B_j-c_j(\Delta)$. Take $h_j$ pairwise vertex-disjoint simplices
$\Sigma_1,\ldots,\Sigma_{h_j}$, each on $j$ new vertices, and form
$\Delta'=\Delta\sqcup\Sigma_1\sqcup\cdots\sqcup\Sigma_{h_j}$.
Each added simplex has one face of cardinality $j$ and none of
cardinality $j+1$. Since $j<\beta\le r$, the new complex is still
$r$-colorable, and
\[
 c_j(\Delta')=B_j,
 \qquad c_{j+1}(\Delta')=c_{j+1}(\Delta).
\]
Only these two face numbers are relevant to the next application of
FFK; the increases in lower face numbers cause no difficulty.

If $s\ge j-1$, the two-term expression defining $B_j$ is its
$(j,r)$-canonical representation: its required inequalities are
$N-\lfloor N/r\rfloor>s$ and $s\ge j-1$. If $s<j-1$, the second
term vanishes and $B_j=\binom Nj_r$ is a one-term canonical
representation, valid because $N\ge\beta\ge j$.
The Frankl--F\"uredi--Kalai theorem, applied to $\Delta'$, therefore
gives
\[
 c_{j+1}(\Delta)=c_{j+1}(\Delta')
 \le\binom N{j+1}_r+\binom{s}{j}_{r-1}=B_{j+1}.
\]
In the one-term case the second term here also vanishes. This
completes the induction, and $c_\beta(\Delta)=k_\beta(H)$ proves the
claim.
\end{proof}

\subsection{Minimum-edge 2-connected graphs}

We use the following large-order specialization of Bougard and
Joret's classification \cite[Proposition~6]{BougardJoret2008}.

\begin{proposition}[Bougard--Joret]\label{prop:BJ}
Let $\alpha\ge3$ and $n\ge3\alpha$. Every $n$-vertex $2$-connected
graph $G$ with $\alpha(G)=\alpha$ satisfies
\begin{equation}\label{eq:BJ-lower}
 e(G)\ge t(n,\alpha)+\alpha,
\end{equation}
with equality if and only if $G$ is a cycle-linking of $T(n,\alpha)$.
\end{proposition}

\begin{proof}
The cited classification gives the stated minimum and three types of
equality cases: cycle-linkings of twisted $T(n,\alpha)$ graphs, odd
subdivisions of $K_4$, and the exceptional graph obtained by pasting
$K_4$ onto $K_4$. Nontrivial twists occur only for
$2\alpha<n<3\alpha$. Odd subdivisions of $K_4$ have
$n-2\alpha=2$, whereas here $n-2\alpha\ge\alpha\ge3$.
The exceptional graph has independence number two. Thus only ordinary
cycle-linkings of $T(n,\alpha)$ remain.
\end{proof}

\section{Cycle-linked Tur\'an graphs and exact counting}
\label{sec:cycle-linkings}

\subsection{The cycle-linked construction}

Throughout the remaining sections, write $n=q\alpha+\rho$, where
$\alpha\ge3$, $q\ge2$, and $0\le\rho<\alpha$. For a cyclic binary
word $w=w_0\cdots w_{\alpha-1}$ with $\rho$ ones, take disjoint
cliques $C_i\cong K_{s_i}$, where $s_i:=q+w_i$. Choose distinct
vertices $x_i,y_i\in C_i$ and add the linking edges
\[
 y_ix_{i+1},\qquad i\in\mathbb Z/\alpha\mathbb Z.
\]
Denote the resulting graph by $C_q(w)$. Since $s_i\ge2$, the
attachment choices are possible. All such choices yield isomorphic
graphs, as do rotations or reversals of $w$.

\begin{figure}[ht]
\centering
\begin{tikzpicture}[scale=.88,every node/.style={font=\small}]
 \node[draw,rounded corners,minimum width=1.45cm,minimum height=.85cm]
   (C1) at (90:3cm) {$K_{q+w_0}$};
 \node[draw,rounded corners,minimum width=1.45cm,minimum height=.85cm]
   (C2) at (30:3cm) {$K_{q+w_1}$};
 \node[draw,rounded corners,minimum width=1.45cm,minimum height=.85cm]
   (C3) at (-30:3cm) {$K_{q+w_2}$};
 \node[draw,rounded corners,minimum width=1.45cm,minimum height=.85cm]
   (C4) at (-90:3cm) {$K_{q+w_3}$};
 \node[draw,rounded corners,minimum width=1.45cm,minimum height=.85cm]
   (C5) at (-150:3cm) {$K_{q+w_4}$};
 \node[draw,rounded corners,minimum width=1.45cm,minimum height=.85cm]
   (C6) at (150:3cm) {$K_{q+w_5}$};
 \draw[thick] (C1)--(C2)--(C3)--(C4)--(C5)--(C6)--(C1);
 \node at (0,0) {cyclic order $w$};
\end{tikzpicture}
\caption{A cycle-linking, shown with six clique components. The two
linking edges incident with each clique use distinct vertices.}
\label{fig:cycle-linking}
\end{figure}
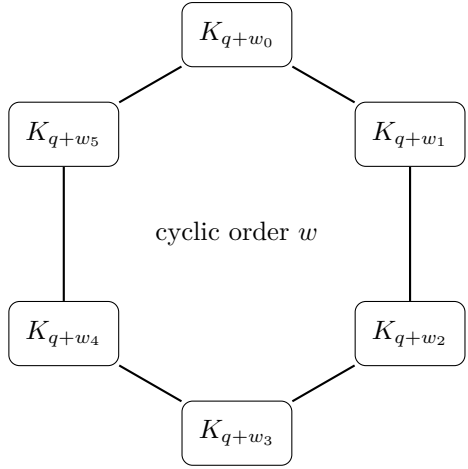

\begin{lemma}\label{lem:construction-admissible}
The graph $C_q(w)$ is $2$-connected and has independence number
$\alpha$.
\end{lemma}

\begin{proof}
Each vertex is incident with at most one linking edge. Deleting a
vertex leaves every clique nonempty and removes at most one linking
edge; the remaining linking edges connect the cliques along a
spanning path. Thus the graph remains connected. An independent set
uses at most one vertex from each clique, while
$\{x_i:0\le i<\alpha\}$ is an independent set of size $\alpha$.
\end{proof}

\begin{lemma}\label{lem:canonical-cycle}
Every cycle-linking of $T(n,\alpha)$ in \cref{prop:BJ} is isomorphic
to $C_q(w)$ for a cyclic word $w$ encoding its clique orders.
\end{lemma}

\begin{proof}
Contract the clique components, retaining the $\alpha$ added edges.
The quotient is a connected loopless multigraph with $\alpha$
vertices and $\alpha$ edges. It has no bridge, since a bridge would
also be a bridge of the original graph. It is therefore a unicyclic
multigraph consisting entirely of its cycle; as $\alpha\ge3$, it is
$C_\alpha$. The two linking edges at a clique must have distinct
endpoints there, since otherwise their common endpoint would be a
cut vertex. Ordering the cliques around the quotient cycle gives the
required representation.
\end{proof}

\subsection{An exact inclusion--exclusion formula}

Let $\cM(F)$ be the set of matchings of a graph $F$, and let
$\cM_j(F)$ consist of its $j$-edge matchings. Write $V(M)$ for the
vertices covered by $M$. For variables $z_1,\ldots,z_m$, let
\[
 e_k(z_1,\ldots,z_m)
 :=\sum_{1\le i_1<\cdots<i_k\le m}z_{i_1}\cdots z_{i_k},
\]
with $e_0=1$ and $e_k=0$ for $k<0$ or $k>m$.

Inclusion--exclusion on the linking edges gives the following
expression for the independence polynomial in terms of matchings
of the quotient cycle.
\begin{proposition}\label{prop:matching-formula}
For every cyclic word $w$ under consideration,
\begin{equation}\label{eq:matching-polynomial}
 I(C_q(w);x)
 =\sum_{M\in\cM(C_\alpha)}(-x^2)^{|M|}
       \prod_{i\notin V(M)}(1+s_ix).
\end{equation}
Consequently, for $0\le\beta\le\alpha$,
\begin{equation}\label{eq:coefficient-matching}
 i_\beta(C_q(w))
 =\sum_{j=0}^{\floor{\beta/2}}(-1)^j
   \sum_{M\in\cM_j(C_\alpha)}
       e_{\beta-2j}(s_i:i\notin V(M)).
\end{equation}
\end{proposition}

\begin{proof}
Let $G_0$ be the disjoint union of the cliques before the linking
edges are added. Then $I(G_0;x)=\prod_i(1+s_ix)$. For
$e_i=y_ix_{i+1}$, let $\mathcal A_i$ consist of the independent sets
of $G_0$ containing both endpoints of $e_i$. The independent sets of
$C_q(w)$ are exactly those belonging to none of these families.

Identify linking edges with the corresponding edges of the quotient
cycle. An intersection of the families $\mathcal A_i$ is nonempty
exactly when the corresponding edges form a matching. Indeed,
adjacent quotient edges force two distinct vertices in their common
clique to be selected. Conversely, the endpoints prescribed by a
matching lie in distinct cliques and can be selected simultaneously.
For a matching $M$, these endpoints fix $2|M|$ vertices, and hence
\[
 \sum_{S\in\bigcap_{e_i\in M}\mathcal A_i}x^{|S|}
 =x^{2|M|}\prod_{i\notin V(M)}(1+s_ix).
\]
For $M=\varnothing$, the intersection means all independent sets of
$G_0$. Weighted inclusion--exclusion now gives
\eqref{eq:matching-polynomial}. Taking the coefficient of $x^\beta$
gives \eqref{eq:coefficient-matching}.
\end{proof}

\section{Optimal cyclic orders}
\label{sec:block-reversals}
\label{sec:local-balance}

Fix $\alpha\ge3$, $q\ge2$, and $0\le\rho<\alpha$. We compare the
cycle-linkings from \cref{sec:cycle-linkings} as their clique-size
words range over cyclic binary words of length $\alpha$ with $\rho$
ones. Our switch is chosen from a shortest imbalance. Its unchanged
boundary segments preserve short-factor weights, while a polynomial
identity determines exactly which independent-set counts increase.

\subsection{Clique-path identities}

For a linear binary word $u=u_1\cdots u_m$, let $L_q(u)$ be the
clique-path with successive clique orders $q+u_1,\ldots,q+u_m$,
using distinct attachment vertices at each internal clique. The empty
word $\varepsilon$ represents the empty graph. Put
$P_u(x):=I(L_q(u);x)$, so $P_\varepsilon(x)=1$.
For nonempty $u$, let ${}^-u$ and $u^-$ denote deletion of its first
and last letters, respectively, and define
\begin{equation}\label{eq:D-def}
 D_u(x):=P_{{}^-u}(x)-P_{u^-}(x),
 \qquad D_\varepsilon(x):=0.
\end{equation}
Write $u^R$ for reversal and call $u$ a \emph{palindrome} if $u=u^R$.
A factor is a contiguous subword; a prefix or suffix is a factor at
the beginning or end of a word.
For $|u|\ge2$, write $u^\circ$ for deletion of both end letters;
thus $P_{u^\circ}=1$ when $|u|=2$. We write $|u|_1$ for the number
of ones in $u$ and occasionally suppress the variable $x$.

We first record the basic path recurrences and positivity properties
needed for the switching argument.
\begin{lemma}
\label{lem:path-facts}
For $b\in\{0,1\}$, one has $P_b(x)=1+(q+b)x$. If $u\ne\varepsilon$,
then
\begin{equation}\label{eq:path-recurrence}
 P_{ub}(x)=\bigl(1+(q+b)x\bigr)P_u(x)-x^2P_{u^-}(x).
\end{equation}
For every word $u$, one has $P_{u^R}=P_u$, and for
$a,b\in\{0,1\}$,
\begin{equation}\label{eq:D-rec}
 D_{aub}(x)=(b-a)xP_u(x)+x^2D_u(x).
\end{equation}
Both $P_u+xD_u$ and $P_u-xD_u$ have constant coefficient $1$,
strictly positive coefficients in all degrees $0,\ldots,|u|$, and
no terms of higher degree.
\end{lemma}

\begin{proof}
The one-letter formula is immediate. Appending a clique to $L_q(u)$
first gives the disjoint-union contribution
$(1+(q+b)x)P_u$. Adding the new linking edge invalidates precisely
the sets containing its two endpoints. The endpoint in the old
terminal clique is distinct from its attachment toward the preceding
clique, when such a clique exists. Thus the remaining choices are
counted by $P_{u^-}$, proving \eqref{eq:path-recurrence}.
Reversing the path gives an isomorphic graph. Applying the recurrence
at both ends yields
\[
 D_{aub}=P_{ub}-P_{au}
        =(b-a)xP_u+x^2(P_{{}^-u}-P_{u^-}).
\]
When $u=\varepsilon$, \eqref{eq:D-rec} follows directly from
$D_{ab}=P_b-P_a=(b-a)x$.

For positivity, assume $u\ne\varepsilon$. Choose a vertex in the last
clique with no neighbor outside that clique; this is possible since
$q\ge2$. The independent sets containing it are counted by
$xP_{u^-}$, so $P_u-xP_{u^-}$ has nonnegative coefficients.
The analogous argument at the other end gives the same conclusion
for $P_u-xP_{{}^-u}$. Now
\[
 \begin{aligned}
 P_u+xD_u&=(P_u-xP_{u^-})+xP_{{}^-u},\\
 P_u-xD_u&=(P_u-xP_{{}^-u})+xP_{u^-}.
 \end{aligned}
\]
Every clique-path has an independent transversal: choose vertices
successively along the path, with at most one vertex forbidden at
each step. Its subsets show that a path with $m$ cliques has a
positive coefficient in every degree $0,\ldots,m$.
Consequently the second summands in the displayed decompositions
supply strict positivity in degrees $1,\ldots,|u|$; the constant
coefficient is $1$. No higher degree is possible, since
$\deg P_u\le |u|$ and $\deg D_u\le |u|-1$.
The empty-word case is immediate from
$P_\varepsilon=1$ and $D_\varepsilon=0$.
\end{proof}

We next record the basic comparison identity for reversing one block
of a cyclic word. It reduces the change in the independence polynomial
to the endpoint asymmetries of the two resulting paths.
\begin{lemma}\label{lem:block-reversal}
If $w=UV$ with $|U|,|V|\ge2$, then
\begin{equation}\label{eq:block-reversal}
 I(C_q(U^RV);x)-I(C_q(UV);x)=-x^2D_U(x)D_V(x).
\end{equation}
\end{lemma}

\begin{proof}
Delete the two linking edges between the $U$- and $V$-blocks, leaving
$L_q(U)\sqcup L_q(V)$. Inclusion--exclusion on the two deleted edges
gives
\begin{equation}\label{eq:cycle-two-boundaries}
 \begin{aligned}
 I(C_q(UV);x)
 ={}&P_UP_V-x^2\bigl(P_{U^-}P_{{}^-V}+P_{{}^-U}P_{V^-}\bigr)\\
 &+x^4P_{U^\circ}P_{V^\circ}.
 \end{aligned}
\end{equation}
The negative terms prescribe the endpoints of one boundary edge;
the positive correction prescribes all four endpoints. The
distinct-attachment convention ensures that the prescribed vertices
have no neighbors in the remaining paths and are mutually independent
before the boundary edges are restored. This also holds for a block
of length two: its internal linking edge uses different attachment
vertices from the prescribed boundary vertices.

Reversing $U$ preserves $P_U$ and $P_{U^\circ}$ and exchanges
$P_{U^-}$ with $P_{{}^-U}$. Subtracting the two versions of
\eqref{eq:cycle-two-boundaries} therefore gives
$-x^2(P_{{}^-U}-P_{U^-})(P_{{}^-V}-P_{V^-})=-x^2D_UD_V$.
\end{proof}

\subsection{Shortest imbalances and the balancing switch}

The palindrome obstruction to binary balance is classical; see the
proof of \cite[Proposition~4.4]{DeLucaZamboni2021}. We include the
argument to obtain the precise length and disjointness needed by our
switch.

\begin{lemma}\label{lem:first-imbalance}
Let $w$ be a nonbalanced cyclic binary word of length $\alpha$, and
let $r$ be the least length at which balance fails. Then
$
2\le r\le\lfloor\alpha/2\rfloor,
$
and $w$ has disjoint cyclic occurrences of $0z0$ and $1z1$, where
$z=z^R$ and $|z|=r-2$. Consequently, for some possibly empty words
$A,B$,
\begin{equation}\label{eq:obstruction-decomposition}
w\equiv 0z0\,A\,1z1\,B\equiv (z0A1z)(1B0).
\end{equation}
\end{lemma}

\begin{proof}
Length-one factors are balanced, so $r\ge2$. Every factor length can
be written as $a\alpha+s$, where $0\le s<\alpha$. Such a factor has
weight $a|w|_1$ plus the weight of its remaining length-$s$ factor.
Hence any imbalance at length $a\alpha+s$ already occurs at length
$s$, and therefore $r<\alpha$. If $r>\alpha/2$, the complementary
intervals of two length-$r$ factors witnessing an imbalance have
length $\alpha-r<r$ and the same absolute weight difference,
contradicting the minimality of $r$. Thus
$
2\le r\le\lfloor\alpha/2\rfloor.
$

Choose length-$r$ factors
$
u=u_1\cdots u_r,$$v=v_1\cdots v_r
$
such that
$
\delta:=|v|_1-|u|_1\ge2.
$
Since $r$ is the least imbalance length, the length-$(r-1)$ factors
obtained by deleting the first letters have weights differing by at
most one. Hence
\[
\delta-(v_1-u_1)
=
|v_2\cdots v_r|_1-|u_2\cdots u_r|_1
\le1.
\]
Since $\delta\ge2$, this gives $v_1-u_1\ge1$. As
$u_1,v_1\in\{0,1\}$, necessarily
$
u_1=0,v_1=1.
$
Substituting back gives $\delta\le2$, and hence $\delta=2$.
Applying the same argument after deleting the last letters yields
$
u_r=0,v_r=1.
$
Thus
$
u=0a0, v=1b1,
$
where $|a|=|b|=r-2$ and $|a|_1=|b|_1$.

We next show that $a=b$. Suppose otherwise, and let $p$ be their
maximal common prefix. If
$
a=p0a', b=p1b',
$
then $0p0$ and $1p1$ are equal-length factors of $w$ whose weights
differ by two, contradicting the minimality of $r$. If instead
$
a=p1a', b=p0b',
$
then $|a|_1=|b|_1$ implies
$
|b'|_1=|a'|_1+1.
$
Consequently, the equal-length factors $a'0$ and $b'1$ have weights
differing by two, again at a length smaller than $r$. Therefore
$a=b=:z$.

We claim that $z$ is a palindrome. Otherwise, comparing $z$ from
the two ends up to the first mismatch gives either
$
z=p0c1p^R
$ or $
z=p1c0p^R
$
for some words $p,c$. In the first case, the prefix $0p0$ of $0z0$
and the suffix $1p^R1$ of $1z1$ have weights differing by two; in
the second case, the same is true of $1p1$ and $0p^R0$. In either
case these factors have length $|p|+2<r$, contradicting the
minimality of $r$. Hence $z=z^R$.

It remains to show that the occurrences of $0z0$ and $1z1$ are
disjoint. They clearly cannot coincide. Suppose they overlap
properly. Since $2r\le\alpha$, there is a position outside their
union; cutting the cyclic word there turns the two occurrences into
overlapping linear intervals. Thus, for some $1\le h<r$, a suffix
of one occurrence equals a prefix of the other.

Suppose first that a length-$h$ suffix of $0z0$ equals a length-$h$
prefix of $1z1$. The former has weight equal to that of the last
$h-1$ letters of $z$, whereas the latter has weight one more than
that of the first $h-1$ letters of $z$. Since $z=z^R$, these two
parts of $z$ have the same weight, a contradiction. The opposite
overlap is excluded in the same way. Hence the two occurrences are
disjoint.

Their disjointness gives
$
w\equiv 0z0\,A\,1z1\,B
$
for some possibly empty words $A,B$. Rotating the cyclic word by
moving its initial $0$ to the end gives
$
w\equiv (z0A1z)(1B0),
$
which proves \eqref{eq:obstruction-decomposition}.
\end{proof}

For a cyclic word, its length-$s$ \emph{factor-weight multiset}
records the number of ones in the factor starting at each position.
Occurrences are counted separately even when the corresponding
factors are equal as words.
The decomposition in Lemma~\ref{lem:first-imbalance} allows us to
turn a shortest imbalance into a block reversal. The next lemma
shows that this reversal produces a precise coefficientwise
improvement.

\begin{lemma}\label{lem:balancing-switch}
Suppose $w=(z0A1z)(1B0)$, where $z=z^R$ and $A,B,z$ may be empty.
Put $r=|z|+2$ and $w^\star=(z1A^R0z)(1B0)$. Then
\begin{equation}\label{eq:special-switch-gain}
 I(C_q(w^\star);x)-I(C_q(w);x)
 =x^{2r}(P_A+xD_A)(P_B-xD_B).
\end{equation}
The difference has strictly positive coefficients in exactly the
degrees $2r,\ldots,\alpha$. Moreover, for every $1\le s<r$, the
length-$s$ factor-weight multisets of $w$ and $w^\star$ coincide.
\end{lemma}

\begin{proof}
Set $k=|z|$, $U=z0A1z$, and $V=1B0$. Since $z=z^R$, we have
$U^R=z1A^R0z$. Both blocks have length at least two, and the operation
reverses the first clique-path before reconnecting the two boundary
edges. It preserves the word length, the number of ones, and the
distinct-attachment convention, so it remains in the same
cycle-linking family.

Equal end letters give $D_{aXa}=x^2D_X$ by \eqref{eq:D-rec}.
Removing the $k$ equal outer pairs of $U$ therefore gives
\[
 D_U=x^{2k+1}(P_A+xD_A),
 \qquad D_V=-x(P_B-xD_B).
\]
Substitution into \eqref{eq:block-reversal} proves
\eqref{eq:special-switch-gain}. By \cref{lem:path-facts}, the two
factors on the right have strictly positive coefficients throughout
degrees $0,\ldots,|A|$ and $0,\ldots,|B|$. Their product is strictly
positive throughout $0,\ldots,|A|+|B|$: in each such degree its
coefficient is a nonempty sum of positive products. Since
$\alpha=2r+|A|+|B|$, the coefficient assertion follows. Empty $A$ or
$B$ is covered by $P_\varepsilon=1$ and $D_\varepsilon=0$.

Fix $1\le s<r$, so $s\le k+1$, and keep the cyclic positions fixed.
The first $k$ and last $k$ positions of $U$ are unchanged: both
segments are $z$ before and after reversal. The whole $V$-block is
also unchanged.

For factors contained in $U$, reflection of their intervals inside
$U$ gives a weight-preserving bijection to factors in $U^R$.
A factor not contained in $U$ cannot contain all of $U$, because
$|U|\ge2k+2>s$. If it meets $U$, it also uses at least one position
in $V$ and hence at most $s-1\le k$ positions in $U$. Its portions
in $U$ form a prefix, a suffix, or their union, and each lies in the
unchanged boundary segments. Thus this factor is unchanged at the
same cyclic positions, even if it contains all of $V$. Reflection on
the first class and the identity on the second prove the multiset
assertion. The argument also covers $k=0$, when $s=1$.
\end{proof}

\subsection{Coefficientwise comparison and equality cases}

Let $b$ be the balanced cyclic word supplied by
\eqref{eq:mechanical-word}. Recall that it is unique up to rotation
\cite[Proposition~4.4 and the preceding discussion]{DeLucaZamboni2021}.
Iterating the switch in \cref{lem:balancing-switch} gives a
coefficientwise comparison with $b$. The least imbalance length
determines exactly where strict inequality begins.

\begin{theorem}
\label{thm:coefficientwise-order}
Every cyclic binary word $w$ of length $\alpha$ with $\rho$ ones
satisfies
\[
 I(C_q(w);x)\preceqcf I(C_q(b);x),
\]
with equality if and only if $w$ is balanced. More precisely, if $w$
is not balanced and $r$ is its least imbalance length, then
\begin{equation}\label{eq:first-imbalance-gap}
 \begin{aligned}
 i_j(C_q(w))&=i_j(C_q(b)) &&(0\le j<2r),\\
 i_j(C_q(w))&<i_j(C_q(b)) &&(2r\le j\le\alpha).
 \end{aligned}
\end{equation}
\end{theorem}

\begin{proof}
The balanced case follows from uniqueness up to rotation. Otherwise
start with $w^{(0)}=w$. Whenever the current word $w^{(i)}$ is not
balanced, take its least imbalance length $r_i$, use
\cref{lem:first-imbalance} to find the corresponding configuration,
and apply \cref{lem:balancing-switch} to obtain $w^{(i+1)}$.
Every step preserves length and weight, does not decrease any
coefficient, and strictly increases $i_\alpha$ because $2r_i\le\alpha$.
There are finitely many cyclic orders of this length and weight, and
strict increase of $i_\alpha$ prevents repetition. The process must
therefore terminate at a balanced word, since otherwise another
switch would be available. By uniqueness, the terminal graph is
isomorphic to $C_q(b)$, proving coefficientwise domination.

A switch at length $r_i$ preserves all shorter-factor weight
multisets. Those lengths were balanced and remain balanced, so
$r_{i+1}\ge r_i$ whenever another step is required. Thus all switch
lengths are at least $r_0=r$, and no step changes a coefficient of
degree below $2r$. The first step strictly increases every coefficient
in degrees $2r,\ldots,\alpha$, and subsequent steps never decrease
them. This proves \eqref{eq:first-imbalance-gap} and excludes equality
for nonbalanced words.
\end{proof}
The precise equality range in \cref{thm:coefficientwise-order}
also determines the maximizing cyclic orders for each fixed
coefficient.
\begin{theorem}\label{thm:fixed-order}
Let $3\le\beta\le\alpha$ and $t=\lfloor\beta/2\rfloor$.
Among all cyclic binary words of length $\alpha$ with $\rho$ ones,
$i_\beta(C_q(w))$ is maximized precisely by the $t$-balanced words.
\end{theorem}

\begin{proof}
By \cref{thm:coefficientwise-order}, the maximum is the value at $b$.
A balanced word is $t$-balanced and attains it. For a nonbalanced word
$w$ with least imbalance length $r$, being $t$-balanced is equivalent
to $r>t$, or equivalently $\beta<2r$. By
\eqref{eq:first-imbalance-gap}, this is precisely the condition for
$i_\beta(C_q(w))=i_\beta(C_q(b))$.
\end{proof}

\section{Excluding extra edges and identifying the extremal graphs}
\label{sec:global-reduction}

Throughout this section, let $3\le\beta\le\alpha$ and put
$p=\beta-2$. We compare the loss caused by the linking edges with
the loss forced by any graph having more than the minimum number
of edges.

\subsection{Comparing independent-set losses}
We first bound from above the number of independent $\beta$-sets
lost when the linking edges are added to $T(n,\alpha)$.
\begin{lemma}\label{lem:cycle-loss}
Every cycle-linking $C_q(w)$ satisfies
\begin{equation}\label{eq:cycle-loss-bound}
 \turcl{\alpha}{\beta}{n}-i_\beta(C_q(w))
 \le\alpha\binom{\alpha-2}{p}(q+1)^p.
\end{equation}
\end{lemma}

\begin{proof}
A linking edge invalidates only the independent sets containing both
its endpoints. An independent $\beta$-set of this kind is completed
by selecting $p$ other cliques and one vertex in each. There are at
most $\binom{\alpha-2}{p}(q+1)^p$ choices. The union bound over the
$\alpha$ linking edges proves the claim.
\end{proof}
We next apply \cref{lem:general-ffk} to the complement of $G$.
This gives a lower bound on the loss relative to $T(n,\alpha)$
when $e(G)\ge t(n,\alpha)+\alpha+1$.
\begin{lemma}\label{lem:ffk-loss}
Suppose $q\ge3$. If $G$ has order $n$, independence number $\alpha$,
and $e(G)\ge t(n,\alpha)+\alpha+1$, then
\begin{equation}\label{eq:ffk-loss-lower}
 \turcl{\alpha}{\beta}{n}-i_\beta(G)
 \ge(\alpha+1)\binom{\alpha-2}{p}(q-2)^p.
\end{equation}
\end{lemma}

\begin{proof}
Put $H=\overline G$ and $d=n-\lceil n/\alpha\rceil$. Then $H$ is
$K_{\alpha+1}$-free, $i_\beta(G)=k_\beta(H)$, and
\cref{lem:turan-recurrence} gives
\[
 e(H)\le\turcl{\alpha}{2}{n}-\alpha-1
       =\turcl{\alpha}{2}{n-1}+(d-\alpha-1).
\]
To apply \cref{lem:general-ffk}, take $N=n-1$ and $s=d-\alpha-1$.
Here $N\ge\beta$, $N-\lfloor N/\alpha\rfloor=d$, and
$d\ge q(\alpha-1)$ gives $s\ge2\alpha-4\ge0$; also $s<d$.
Consequently
\[
 i_\beta(G)\le\turcl{\alpha}{\beta}{n-1}
                  +\turcl{\alpha-1}{\beta-1}{d-\alpha-1}.
\]
Combining this bound with \cref{lem:turan-recurrence} at order $n$
yields
\begin{equation}\label{eq:ffk-loss-first}
 \turcl{\alpha}{\beta}{n}-i_\beta(G)
 \ge\turcl{\alpha-1}{\beta-1}{d}
       -\turcl{\alpha-1}{\beta-1}{d-\alpha-1}.
\end{equation}

To bound this difference, start with $\Turan_{\alpha-1}(d)$ and
delete $\alpha+1$ vertices successively, each time from a largest
part. Every intermediate graph is balanced, and the final graph is
$\Turan_{\alpha-1}(d-\alpha-1)$. Since
\[
 d-\alpha-1\ge q(\alpha-1)-(\alpha+1)
             \ge(q-2)(\alpha-1),
\]
every part has at least $q-2$ vertices throughout. Each deleted
vertex therefore belongs to at least
$\binom{\alpha-2}{p}(q-2)^p$ copies of $K_{\beta-1}$ in the graph
present at that step. Each lost clique is counted only at the first
deletion of one of its vertices. Summing the $\alpha+1$ successive
losses and using \eqref{eq:ffk-loss-first} proves
\eqref{eq:ffk-loss-lower}.
\end{proof}

\subsection{Proofs of the main theorems}

\begin{proof}[Proof of \cref{thm:fixed-size}]
We show that $N_{\alpha,\beta}:=3\alpha(\beta-2)(\alpha+1)$ suffices.
For $n\ge N_{\alpha,\beta}$, one has $q\ge3p(\alpha+1)$ and hence
$n\ge3\alpha$. Bernoulli's inequality gives
\[
 \left(\frac{q-2}{q+1}\right)^p
 =\left(1-\frac3{q+1}\right)^p
 \ge1-\frac{3p}{q+1}>\frac{\alpha}{\alpha+1},
\]
where strictness follows from $q+1>3p(\alpha+1)$. Thus
$(\alpha+1)(q-2)^p>\alpha(q+1)^p$.

Fix $C\in\mathcal B_\beta(n,\alpha)$; such a graph exists because
a balanced word is $\lfloor\beta/2\rfloor$-balanced. If an admissible
graph $G$ has at least $t(n,\alpha)+\alpha+1$ edges, then
\cref{lem:ffk-loss,lem:cycle-loss} give
\[
 \begin{aligned}
 \turcl{\alpha}{\beta}{n}-i_\beta(G)
 &\ge(\alpha+1)\binom{\alpha-2}{p}(q-2)^p\\
 &>\alpha\binom{\alpha-2}{p}(q+1)^p
 \ge\turcl{\alpha}{\beta}{n}-i_\beta(C).
 \end{aligned}
\]
Hence $i_\beta(G)<i_\beta(C)$. By \cref{prop:BJ}, every remaining
admissible graph has exactly $t(n,\alpha)+\alpha$ edges and is a
cycle-linking of $T(n,\alpha)$. Its representation $C_q(w)$ follows
from \cref{lem:canonical-cycle}, and \cref{thm:fixed-order} gives the
required inequality and equality characterization within this family.
\end{proof}

\begin{proof}[Proof of \cref{thm:main-coefficientwise}]
Take
\[
 N_\alpha:=\max\left\{3\alpha,
                 \max_{3\le\beta\le\alpha}N_{\alpha,\beta}\right\},
\]
using the thresholds in the preceding proof. This is finite and is
$O(\alpha^3)$. For $n\ge N_\alpha$, \cref{thm:fixed-size} applies to
every $3\le\beta\le\alpha$. Since the clique-size word of
$B(n,\alpha)$ is balanced, it satisfies every required local balance
condition, and $i_\beta(G)\le i_\beta(B(n,\alpha))$ for all these
$\beta$. Also $i_0=1$, $i_1=n$, and $i_j=0$ for $j>\alpha$ for both
graphs. Finally, \cref{prop:BJ} and $i_2(G)=\binom n2-e(G)$ give the
inequality for $i_2$. This proves coefficientwise domination.

If equality holds as polynomials, equality at $i_\alpha$ and
\cref{thm:fixed-size} imply that $G\cong C_q(w)$ for a
$\lfloor\alpha/2\rfloor$-balanced word $w$. By
\cref{lem:first-imbalance}, such a word is balanced. Uniqueness of the
balanced cyclic order and equivalence of attachment choices therefore
give $G\cong B(n,\alpha)$. The converse is immediate.
\end{proof}

\section{Special cases and examples}
\label{sec:examples}

\subsection{Equal clique orders}

When all clique orders are equal, \cref{prop:matching-formula} and the
matching counts for a cycle give the following formula.

\begin{corollary}\label{cor:divisible-value}
If $n=\alpha q$ and $0\le\beta\le\alpha$, then
\begin{equation}\label{eq:divisible-value}
 i_\beta(C_q(0^\alpha))
 =\sum_{j=0}^{\floor{\beta/2}}(-1)^j
   \frac{\alpha}{\alpha-j}\binom{\alpha-j}{j}
   \binom{\alpha-2j}{\beta-2j}q^{\beta-2j}.
\end{equation}
\end{corollary}

\begin{proof}
A path on $m$ vertices has $\binom{m-j}{j}$ matchings of size $j$:
if $i_1<\cdots<i_j$ are the selected edge indices, the map
$(i_1,\ldots,i_j)\mapsto(i_1,i_2-1,\ldots,i_j-j+1)$ is a bijection
to the $j$-subsets of $\{1,\ldots,m-j\}$. For $j\ge1$, splitting
cycle matchings according to a fixed edge gives
\[
 |\cM_j(C_\alpha)|
 =\binom{\alpha-j}{j}+\binom{\alpha-j-1}{j-1}
 =\frac{\alpha}{\alpha-j}\binom{\alpha-j}{j}.
\]
The final expression is also $1$ when $j=0$. Each matching leaves
$\alpha-2j$ cliques of order $q$, so its contribution in
\eqref{eq:coefficient-matching} is
$(-1)^j\binom{\alpha-2j}{\beta-2j}q^{\beta-2j}$. Summing proves
the formula.
\end{proof}

\subsection{Independent triples}

Every cyclic binary word is $1$-balanced, so every cycle-linking has
the same number of independent triples.

\begin{corollary}\label{cor:beta3}
If $q\ge3(\alpha+1)$, then every $n$-vertex $2$-connected graph $G$
with $\alpha(G)=\alpha$ satisfies
\[
 i_3(G)\le i_3(T(n,\alpha))-(\alpha-2)n,
\]
with equality exactly for the cycle-linkings of $T(n,\alpha)$.
\end{corollary}

\begin{proof}
Only matchings of sizes zero and one contribute to
\eqref{eq:coefficient-matching}. The edge between $C_i$ and $C_{i+1}$
invalidates $n-s_i-s_{i+1}$ triples, and
$\sum_i(n-s_i-s_{i+1})=(\alpha-2)n$. Apply
\cref{thm:fixed-size} with $\beta=3$ and the threshold from its proof.
\end{proof}

\subsection{Adjacency and higher-order balance}

For $\beta\in\{4,5\}$ with $\beta\le\alpha$, Theorem~4.6 identifies
the maximizing cyclic words as precisely the $2$-balanced ones.
A binary cyclic word is $2$-balanced if and only if $00$ and $11$
do not both occur. Thus no two large cliques are consecutive when
$\rho\le\alpha/2$, and no two small cliques are consecutive when
$\rho\ge\alpha/2$.

For $\beta=4$, Proposition~3.3 makes the dependence on adjacency
explicit:
\begin{equation}
\begin{aligned}
 i_4(C_q(w))
 &=e_4(s_0,\ldots,s_{\alpha-1})
   -(\alpha-4)e_2(s_0,\ldots,s_{\alpha-1})\\
 &\quad-\sum_{i=0}^{\alpha-1}s_i s_{i+1}
   +\frac{\alpha(\alpha-3)}2,
\end{aligned}
\end{equation}
where indices are taken modulo $\alpha$.
Indeed, a pair of clique indices is disjoint from $\alpha-4$
cycle edges, with one additional edge when the pair is consecutive,
and the cycle has $\alpha(\alpha-3)/2$ two-edge matchings.
Only the adjacent-product sum depends on the cyclic order. Since
\[
 \sum_{i=0}^{\alpha-1}s_i s_{i+1}
 =\alpha q^2+2q\rho
  +\sum_{i=0}^{\alpha-1}w_iw_{i+1},
\]
maximizing $i_4$ is equivalent to minimizing the number of cyclic
occurrences of $11$, in agreement with the $2$-balance criterion.

However, an order that maximizes $i_4$ and $i_5$ need not maximize
higher coefficients. Take $\alpha=6$ and $\rho=2$, and consider
\[
 w=101000,\qquad b=100100.
\]
Both words are $2$-balanced. The word $w$ contains the length-three
factors $101$ and $000$, so its least imbalance length is three,
whereas every length-three factor of $b$ contains exactly one $1$.
Applying Lemma~4.4 with $z=0$ and $A=B=\varepsilon$, up to rotation,
gives
\[
 I(C_q(b);x)-I(C_q(w);x)=x^6.
\]
Thus the two graphs have the same number of independent $j$-sets
for every $j\le5$, but $C_q(b)$ has exactly one more independent
$6$-set. This illustrates the distinction between separating the
large cliques and spacing them evenly: both orders maximize $i_4$
and $i_5$, but only the balanced order maximizes $i_6$.

\section{Concluding remarks and open problems}

The large-order threshold comes from comparing the bounds in
\cref{lem:ffk-loss,lem:cycle-loss}. It does not use $2$-connectivity
inside the clique-counting step and is not expected to be sharp.
In contrast, \cref{thm:coefficientwise-order} holds for every $q\ge2$:
within the cycle-linkings of $T(n,\alpha)$, the shortest imbalance
determines the first coefficient that is smaller than the balanced
value. This phenomenon may also be useful in extremal problems with
related generating functions.

For $2\le\alpha<n$, combining \cref{lem:general-ffk} with the
connected minimum-edge bound
$e(G)\ge t(n,\alpha)+\alpha-1$
from~\cite[Proposition~5]{BougardJoret2008} gives, with
\[
 N=n-1,\qquad
 s=n-\left\lceil\frac n\alpha\right\rceil-\alpha+1,
\]
\[
 i_j(G)\le
 \turcl{\alpha}{j}{n-1}
 +\turcl{\alpha-1}{j-1}{s}
 =i_j(TC_{n,\alpha})
 \qquad(2\le j\le\alpha).
\]
Together with the trivial case $\alpha=1$, this recovers the
coefficientwise upper bound of Lehner and Wagner
~\cite[Theorem~2]{LehnerWagner2017}, without a large-order assumption.
Moreover, in the range
$3\le\beta\le\min\{\alpha,n-\alpha\}$, equality in the coefficient
of $x^\beta$ can also be recovered within the same framework.
The lower-shadow form of the Frankl--F\"uredi--Kalai theorem
propagates equality down to $i_2$ and $i_3$; the connected
minimum-edge classification, together with the resulting equality
in the number of independent triples, then identifies
$TC_{n,\alpha}$ as the unique equality case.

In the $2$-connected setting, it remains to understand how far
the large-order assumption can be weakened and which graphs are
extremal at smaller orders. We conclude with the following
related problems.
\begin{problem}
Determine the smallest $N_{\alpha,\beta}$ for which
\cref{thm:fixed-size} holds for every $n\ge N_{\alpha,\beta}$.
\end{problem}

\begin{problem}
Develop a Frohmader--FFK-type clique bound incorporating the condition
that the complement is $2$-connected. Can such a refinement reduce the
present quadratic dependence on $\alpha$ to a linear one for fixed
$\beta$?
\end{problem}

\begin{problem}
Determine the maximum number of independent $\beta$-sets and all
extremal graphs in the remaining small-order range, where twisted
Tur\'an graphs and other structures may compete with ordinary
cycle-linkings.
\end{problem}

\bibliographystyle{unsrt}
\bibliography{references}

@article{BruyereMelot2009,
  author  = {Bruy\`ere, V\'eronique and M\'elot, Hadrien},
  title   = {{Fibonacci} index and stability number of graphs:
             a polyhedral study},
  journal = {Journal of Combinatorial Optimization},
  volume  = {18},
  number  = {3},
  pages   = {207--228},
  year    = {2009},
  doi     = {10.1007/s10878-009-9228-7}
}

@article{MohrRautenbach2021,
  author  = {Mohr, Elena and Rautenbach, Dieter},
  title   = {On the maximum number of maximum independent sets
             in connected graphs},
  journal = {Journal of Graph Theory},
  volume  = {96},
  pages   = {510--521},
  year    = {2021},
  doi     = {10.1002/jgt.22629}
}

@article{Turan1941,
  author  = {Tur{\'a}n, P{\'a}l},
  title   = {Egy gr{\'a}felm{\'e}leti
             sz{\'e}ls{\H{o}}{\'e}rt{\'e}kfeladatr{\'o}l},
  journal = {Matematikai {\'e}s Fizikai Lapok},
  volume  = {48},
  pages   = {436--452},
  year    = {1941},
  note    = {In Hungarian}
}

@book{Bollobas1978,
  author    = {Bollob{\'a}s, B{\'e}la},
  title     = {Extremal Graph Theory},
  series    = {London Mathematical Society Monographs},
  number    = {11},
  publisher = {Academic Press},
  address   = {London},
  year      = {1978}
}

@book{Lothaire2002,
  author    = {Lothaire, M.},
  title     = {Algebraic Combinatorics on Words},
  series    = {Encyclopedia of Mathematics and its Applications},
  volume    = {90},
  publisher = {Cambridge University Press},
  address   = {Cambridge},
  year      = {2002},
  doi       = {10.1017/CBO9781107326019}
}

@article{AlonShikhelman2016,
  author  = {Alon, N. and Shikhelman, C.},
  title   = {Many {$T$} copies in {$H$}-free graphs},
  journal = {J. Combin. Theory Ser. B},
  volume  = {121},
  year    = {2016},
  pages   = {146--172},
  doi     = {10.1016/j.jctb.2016.03.004},
  note    = {\doi{10.1016/j.jctb.2016.03.004}}
}

@article{Zykov1949,
  author  = {Zykov, A. A.},
  title   = {On some properties of linear complexes},
  journal = {Mat. Sbornik N.S.},
  volume  = {24(66)},
  year    = {1949},
  pages   = {163--188},
  note    = {In Russian; English translation in \emph{Amer. Math. Soc. Transl.},
             No.~79 (1952), 1--33}
}

@article{BougardJoret2008,
  author  = {Bougard, N. and Joret, G.},
  title   = {{Tur\'an}'s theorem and {$k$}-connected graphs},
  journal = {J. Graph Theory},
  volume  = {58},
  year    = {2008},
  pages   = {1--13},
  doi     = {10.1002/jgt.20289},
  note    = {\doi{10.1002/jgt.20289}}
}

@misc{DeLucaZamboni2021,
  author        = {De Luca, A. and Zamboni, L. Q.},
  title         = {On the extremal values of the cyclic continuants of
                   {Motzkin} and {Straus}},
  howpublished  = {arXiv:2108.11313 [math.CO]},
  year          = {2021},
  eprint        = {2108.11313},
  archivePrefix = {arXiv},
  primaryClass  = {math.CO}
}

@article{FranklFurediKalai1988,
  author  = {Frankl, P. and F{\"u}redi, Z. and Kalai, G.},
  title   = {Shadows of colored complexes},
  journal = {Math. Scand.},
  volume  = {63},
  year    = {1988},
  pages   = {169--178},
  doi     = {10.7146/math.scand.a-12231},
  note    = {\doi{10.7146/math.scand.a-12231}}
}

@article{Frohmader2008,
  author  = {Frohmader, A.},
  title   = {Face vectors of flag complexes},
  journal = {Israel J. Math.},
  volume  = {164},
  year    = {2008},
  pages   = {153--164},
  doi     = {10.1007/s11856-008-0024-3},
  note    = {\doi{10.1007/s11856-008-0024-3}}
}

@article{LehnerWagner2017,
  author  = {Lehner, F. and Wagner, S.},
  title   = {Maximizing the number of independent sets of fixed size in
             connected graphs with given independence number},
  journal = {Graphs Combin.},
  volume  = {33},
  year    = {2017},
  pages   = {1103--1118},
  doi     = {10.1007/s00373-017-1825-0},
  note    = {\doi{10.1007/s00373-017-1825-0}}
}

\end{document}